\documentclass[a4paper,11pt]{amsart}

\usepackage{amsthm}
\usepackage[utf8]{inputenc}
\usepackage[T1]{fontenc}
\usepackage{amsmath}
\usepackage{graphicx}
\usepackage{epsf, subfigure, verbatim}
\usepackage{epsfig}
\usepackage{amssymb}
\usepackage{amsfonts}
\usepackage{enumerate}
\usepackage{srcltx}
\usepackage{hyperref}
\usepackage{float}

\usepackage{srcltx}

\theoremstyle{plain} 
\newtheorem{definition}[equation]{Definition}
\newtheorem{remark}[equation]{Remark}
\newtheorem{theorem}[equation]{Theorem}
\newtheorem{conjecture}[equation]{Conjecture}
\newtheorem{corollary}[equation]{Corollary}
\newtheorem{proposition}[equation]{Proposition}
\newtheorem{lemma}[equation]{Lemma}
\numberwithin{equation}{subsection}

\begin{document}

\title[Growth rates of Laplace eigenfunctions on the unit disk]{Growth rates of Laplace eigenfunctions on the unit disk}
\author{Guillaume Lavoie and Guillaume Poliquin}
\let\thefootnote\relax
\thanks{Research of the second author is supported by a FRQNT grant.}
\begin{abstract}
We give a description of the growth rates of $L^2$-normalized Laplace eigenfunctions on the unit disk with Dirichlet and Neumann boundary conditions. In particular, we show that the growth rates of both Dirichlet and Neumann eigenfunctions are bounded away from zero. Our approach starts with P. Sarnak growth exponents and uses several key asymptotic formulas for Bessel functions or their zeros. 
\end{abstract}

\maketitle
\section{Introduction and main results}

\subsection{Introduction}

The $L^2$-normalized eigenfunctions $F_{n,m}$ and eigenvalues $\lambda_{n,m}$ of the Dirichlet Laplace operator on $\mathbb{D}$, the unit disk,
\begin{align*}
\Delta F_{n,m} =&  -\lambda_{n,m} F_{n,m}, \medskip \text{ in } \mathbb{D}, \\
F_{n,m} &= 0, \medskip \text{ on } \partial\mathbb{D}, \notag
\end{align*}
form a family parametrized by the order of the associated Bessel function of the first type $J_n$ as well as a choice of the $m$-th positive root of $J_n$.  There are thus many ways to let $\lambda_{n,m} \rightarrow \infty$. We study the non trivial relationship between the growth of the $L^\infty$ norm of  such eigenfunctions and the nature of the subsequence that is used to reach the high frequency limit. 

The idea to do so is motivated by P. Sarnak's letter \cite{Sa1} to C. Morawetz in which P. Sarnak presents an approach for studying the $L^\infty$ norms of families of eigenfunctions in relation to $\lambda$ in a more general setting, that is, if you let $f_\lambda$ be an $L^2$-normalized eigenfunction of the Laplace-Beltrami operator, $\Delta_g f_\lambda = -\lambda f_\lambda,$ on a compact Riemannian boundaryless manifold $(M^d,g)$ with $d\geq 2$. 

Proving bounds on the $L^\infty$ norms of Laplace-Beltrami eigenfunctions has always been a subject of interest. For instance, it is well known that 
\begin{equation}\label{Avakumovic}
\exists C(M,g)>0, ||f_\lambda||_\infty \leq C(M,g) \lambda ^{\frac{d-1}{4}},
\end{equation}
and that the inequality is sharp for $M=S^d$, the round sphere, with the eigenfunctions saturating the upper bound being zonal harmonics (see for instance \cite{So} for details). However,  the growth rate of the $L^\infty$ norms of such eigenfunctions may be improved under additional hypotheses. For instance, in 1985, P. Sarnak conjectured in \cite{Sa2} that, for surfaces of negative curvature, $||f_\lambda||_\infty \leq C(M,g) \lambda^{\epsilon}$ for all $\epsilon > 0$. If $\partial M \neq \emptyset$, in the case of Dirichlet or Neumann boundary conditions, the bound stated in \eqref{Avakumovic} also holds as shown by Grieser in \cite{Gr}. In \cite{TZ1, TZ2}, Toth and Zelditch studied $L^p$ norms of eigenfunctions in the completely integrable case.

In \cite[p. 41]{Sa1}, P. Sarnak defines the $L^\infty$ exponents growth of eigenfunctions on $(M^d,g)$ as the set $E(M)$ of accumulation points of the numbers given by
\begin{equation}\label{Manifold}
\dfrac{\log{||f_\lambda||_{\infty}}}{\log{\lambda}}.
\end{equation}
For $S^2$, Vanderkam showed that there exist some orthonormal bases of eigenfunctions of $\Delta_g$ on $S^2$ with $||f_\lambda||_\infty \ll \log(\lambda)$ (see \cite[Theorem 1.2 and its proof]{V}). Together with 
 \eqref{Avakumovic}, it follows that $E(S^2) \subseteq [0, \lambda^{\frac{1}{4}}]$.
 
In this paper, we investigate the growth rate of the $L^\infty$ norm as $\lambda$ tends to infinity with \eqref{Manifold} on the unit disk exclusively. It is well known that this setting modelizes the vibration of a drum membrane: the eigenvalues are linked to the harmonics and the eigenfunctions to the vibration of the membrane. In the case of a quantum particle, the eigenvalues represent the various energy levels associated with the eigenfunctions. The eigenfunctions are probability density functions used to compute the probability that, given an energy level $\lambda$, a quantum particle is located at a specific point in space. In both cases, $\lambda$ is associated with the energy of the system and the system's high energy limit is obtained by letting $\lambda$ tend to infinity.

\subsection{Main results}

The exponents of the $L^\infty$ growth of Dirichlet eigenfunctions on $\mathbb{D}$ are defined as the set $E^D(\mathbb{D})$ of accumulation points of the numbers
\begin{equation}\label{eAccumulationPts}
\dfrac{\log{||F_{n,m}||_{\infty}}}{\log{\lambda_{n,m}}},
\end{equation}
where $||F_{n,m}||_\infty$ denotes the $L^\infty$ norm of the Dirichlet eigenfunction $F_{n,m}$. 

We show that there exist non-trivial universal lower bounds for the $L^\infty$ norm of Dirichet eigenfunctions by proving the following:

\begin{theorem}\label{thm1}
We have that
$$ [7/36, 1/4] \subseteq E^D(\mathbb{D}) \subseteq [1/18, 1/4].$$
\end{theorem}

Note that the upper bound follows from \cite[Theorem 1]{Gr} and it is sharp on $\mathbb{D}$. As an immediate consequence, we get the following corollary:

\begin{corollary}\label{cor1}
There exists $C>0$ such that
$$||F_{n,m}||_\infty \geq C \lambda_{n,m}^{1/18} + o(\lambda_{n,m}^{1/18}).$$
\end{corollary}
The bound obtained in Corollary \ref{cor1} is not sharp (see Remark \ref{remark1} for more details). At first glance, it is surprising to see that the lowest possible growth rate is bounded away from zero, in contrast with the case of the round sphere. This can be explained by the fact that it is possible to fabricate cancellations using the high degeneracy of the eigenvalues of $S^2$, a property that does not hold in our case.

In 2013, the equivalent of the Bourget conjecture for the Neumann was proved in \cite{As}. In particular, it implies that the dimension of the eigenspace is bounded by $2$. This being the case, we are able to prove the following:
\begin{theorem}\label{thm3}
Let $E^N(\mathbb{D})$ be the set of accumulation points defined by \eqref{eAccumulationPts} where we replace $F_{n,m}$ and $\lambda_{n,m}$ by their Neumann counterparts. We have that
$$ E^N(\mathbb{D}) \subseteq \left[\tfrac{1}{12}, \tfrac{1}{4}\right]. $$
\end{theorem}

\begin{corollary}\label{cor1N}
There exists $C>0$ such that
$$||F_{n,m}||_\infty \geq C \lambda_{n,m}^{1/12} + o(\lambda_{n,m}^{1/12}).$$
\end{corollary}
The bound obtained in Corollary \ref{cor1N} is not sharp.

\subsection{Structure of the paper}

In Section \ref{defi}, we present key definitions and lemmas required in proving the main results. In Section \ref{sprelim}, we state and sometimes prove several propositions regarding Bessel functions and their zeros, most of which can be found in the literature.  In Section \ref{sproofs}, we prove the results regarding the Dirichlet case. We start by proving Theorem \ref{thm2}. We then prove Lemma \ref{lemme421}, Lemma \ref{lemme441} and Proposition \ref{pgamma}. In Section \ref{sproofsNeumann}, we focus on the Neumann case. To do so, we prove Theorem \ref{thm4}, Lemma \ref{lNeumann} and Theorem \ref{thm3}.  

\subsection{Acknowledgements}
The authors are grateful to Iosif Polterovich for precious comments on various results and for useful discussions. The authors are also thankful to Guillaume Roy-Fortin for his comments on several key parts of this article. This work is based on the Master's thesis of Guillaume Lavoie, \cite{Lav}, which was done under the supervision of Iosif Polterovich at the Université de Montréal. The authors express their gratitude to Samtou Bodjona, Dominic Leroux and Émile Roy for their help with the numerical simulations.

\section{Definitions and other key results}\label{defi}

\subsection{The Dirichlet case}

 Given the invariance of the Laplace operator, we can use the identity $A\cos(\beta)+B\sin(\beta) = \sqrt{A^2+B^2} \cos(\beta + \beta_0)$, where $\beta_0$ is a phase shift, to get the following explicit $L^2$-normalized solutions: 
\begin{equation}\label{DirichletEF}
F_{n,m}(r,\theta) = \begin{cases}
\sqrt{\dfrac{1}{\pi}} \cdot  \dfrac{J_n(k_{n,m}r)}{J_{n+1}(k_{n,m})};\; m\in \mathbb{N}, n = 0, \\
\sqrt{\dfrac{2}{\pi}} \cdot  \dfrac{J_n(k_{n,m}r)}{J_{n+1}(k_{n,m})}\cos{n\theta} ;\; m\in \mathbb{N}, n \in \mathbb{N},
\end{cases}
\end{equation}
where $J_n$ is the Bessel function of the first kind of order $n$ and $k_{n,m}$ is the $m$-th root of $J_n$. The eigenvalues are given by $\lambda_{n,m} = k_{n,m}^2$ (see, for instance, \cite{CH, H} for details). 

In order to understand how $\lambda_{n,m}$ increases to infinity, one must take into account how both parameters $n$ and $m$ vary since the sequence $\{\lambda_{n,m}\}$ strictly increases in $n$ and in $m$.  To do so, we define the following:
\begin{definition}\label{dgamma}
Let $l\in\mathbb{N}$. Consider $\{\lambda_{n_l,m_l}\}$, a subsequence of eigenvalues such that $\lambda_{n_l,m_l} < \lambda_{n_{l+1},m_{l+1}}$. If $\gamma_l$ is defined as
\begin{equation}\label{eqgammal}
\dfrac{\log{m_l}}{\log{n_l}} = \gamma_l,
\end{equation}
we denote the set of accumulation points of the numbers $\gamma_l$ by $\Gamma$. Moreover, we let
$$\lim_{l\to\infty} \gamma_l = \gamma,$$
whenever the limit exists.
\end{definition}
When the limit does exist, it follows immediately that $\gamma \in [0,\infty]$. In this case, we abuse the notation by letting $\gamma_l = \gamma$ instead of writing $\gamma_l \to \gamma$.

For an arbitrary subsequence $\{\lambda_{n_l,m_l}\}$ as defined in Definition \ref{dgamma} and such that $\lambda_{n_l,m_l} \to \infty$ as $l\to\infty$, the limit of $\gamma_l$ as $l\to\infty$ generally does not exist. The relation $\phi$ defined below takes that fact into account.
\begin{definition}\label{dphi}
Let $\{\lambda_{n_l,m_l}\}$ be as defined in Definition \ref{dgamma}, $\gamma_l \to \gamma$, and $\{F_{n_l,m_l}\}$ be the associated subsequences of eigenfunctions. We define the relation $\phi$ as follows:
\begin{equation*}
\phi(\gamma) = \liminf_{l\to\infty} \dfrac{\log ||F_{n_l,m_l}||_\infty}{\log \lambda_{n_l,m_l}}.
\end{equation*}
\end{definition}
To prove Theorem \ref{thm1}, it is necessary to understand the behavior of the relation $\phi$.  The next three results shed light on this.

\begin{lemma}\label{lemme421}
If $\gamma < 1$, then
$$\phi(\gamma) \geq \max \left\{ \frac{1 - \gamma}{6} , \frac{\gamma}{12} \right\}.$$
\end{lemma}

For $\gamma>3$, we get an explicit formula for $\phi$. 

\begin{proposition}\label{pgamma}
 If $\gamma >3 $, we then have that
$$\phi(\gamma) = \frac{1}{4} - \frac{1}{6\gamma}.$$
\end{proposition}

For $\gamma\in[1,3)$, we could only prove a weaker result, stated in the following lemma:

\begin{lemma}\label{lemme441}
If $\gamma \geq 1$, we have that
$$\phi(\gamma) \geq \dfrac{1}{12}.$$
\end{lemma}
The bound obtained in Lemma \ref{lemme441} is not sharp. Numerical simulations lead to the following conjecture:
\begin{conjecture}\label{conj12}
 If $1 \leq \gamma <3 $, we have that
$$\phi(\gamma) = \frac{1}{4} - \frac{1}{6\gamma}.$$
\end{conjecture}
The simulations were made using Mathematica 11.3\copyright. Details can be found in Section~\ref{Numeriques}.

Studying the size of $||F_\lambda||_\infty$ gives a certain portrait of the distribution of $F_\lambda$. Since $||F_\lambda||_2=1$, if $||F_\lambda||$ is small on a sufficiently large part of $\mathbb{D}$, then the function is concentrated in a small band around the boundary. This phenomenon is called a whispering gallery and the next theorem and its proof illustrate it.

\begin{theorem}\label{thm2}
Consider the subsequence of $\{ \lambda_{n,m} \}$ defined by $m= \lfloor n^\gamma \rfloor$. If $0 \leq \gamma < 1$, then there exist sequences $\{\rho_n\},\; 0 < \rho_n < 1,\, \rho_n \nearrow 1$ and $\{\epsilon_n\},\; \epsilon_n > 0,\, \epsilon_n \searrow 0$, such that, for all $n \in \mathbb{N}\cup\{0\}$, 
$$|F_{n,m}(r,\theta)| < \epsilon_n,\; (r, \theta) \in [0, \rho_n) \times [0, 2\pi).$$
\end{theorem}
This immediately implies that, under the above hypothesis, $ {|F_{m,n}(r,\theta)|\rightarrow 0}$ on $\mathbb{D}$, but the convergence is not uniform in $r$.

\subsection{The Neumann case}

We also consider the Neumann eigenvalue problem. Denote by $F^N_{n,m}$ the $L^2$-normalized eigenfunctions of the Neumann Laplace operator,
\begin{equation}\label{eEigenfunctionsNeumann}
F^N_{n,m}(r,\theta) =\begin{cases}
\sqrt{\dfrac{1}{\pi \cdot  \left(1-\dfrac{n^2}{k^{\prime \quad 2}_{n,m}}\right)}} \cdot \dfrac{J_n(k'_{n,m} r)}{J_n(k'_{n,m})};\; m\in \mathbb{N}, n = 0, \\
   \sqrt{\dfrac{2}{\pi \left(1-\dfrac{n^2}{k^{\prime \quad 2}_{n,m}}\right)}} \cdot \dfrac{J_n(k'_{n,m} r)}{J_n(k'_{n,m})} \cos(n\theta) ;\; m\in \mathbb{N}, n \in \mathbb{N},
\end{cases}
\end{equation}
and by $\lambda^N_{n,m} = k^{\prime \quad 2}_{n,m}$ where $k_{n,m}'$ is the $m$-th zero of $J_n'$.

When comparing the eigenfunctions of the Neumann problem \eqref{eEigenfunctionsNeumann} with the eigenfunctions of the Dirichlet problem \eqref{DirichletEF}, we see that instead of normalizing with a constant factor multiplied by $1/J_{n+1}(k_{n,m})$, we have a non constant factor that depends on $n$ and $k'_{n,m}$ multiplied by $1/J_n(k'_{n,m})$. The fact that we see the Bessel function $J_n$ instead of $J_{n+1}$ does not impact the ideas used in the proofs of the various lemmas in the Dirichlet case since we use asymptotic formulas for large values of $n$. Also, we can easily bound $k'_{n,m}$ in terms of $n, m$ just as we do for the Dirichlet case. Nevertheless, the fact that it depends on $n$ and on $k'_{n,m}$ forces the use of a different approach to tackle the problem when $n\approx k'_{n,m}$, i.e., when $\gamma < 1$. 

This leads to the following lemma:

\begin{lemma}\label{lNeumann}
Let $\phi^N$ be defined as in Definition \ref{dphi} while replacing every occurence of $F_{n,m}$ and $\lambda_{n,m}$ by $F^N_{n,m}$ and $\lambda^N_{n,m}$ respectively. We have that
\begin{equation}\label{eBornesNeumann}
\phi^N(\gamma) \geq \begin{cases} \dfrac{2-\gamma}{12} &\quad \mbox{ if } \gamma < 1, \\ \dfrac{1}{12} &\quad \mbox{ if } \gamma \geq 1. \end{cases}
\end{equation}

\end{lemma}
It is worth mentioning that  we could not obtain a precise formula for $\gamma>3$ as we had in Proposition \ref{pgamma}. The observation used in the proof that simplified \eqref{eqEstimZeros} when $\gamma>3$ is no longer meaningful in the Neumann case. 

Numerical simulations lead to the following conjecture:
\begin{conjecture}\label{conj123}
If $\gamma\geq1$, we have that
$$\phi^N(\gamma) = \frac{1}{4} - \frac{1}{6\gamma}.$$

\end{conjecture}

We can also obtain an analogous result to Theorem \ref{thm2}, stated as follows:

\begin{theorem}\label{thm4}
Consider the subsequence of $\{ \lambda^N_{n,m} \}$ defined by $m= \lfloor n^\gamma \rfloor$. If $0 \leq \gamma < 1$, then there exist sequences $\{\rho_n\},\; 0 < \rho_n < 1,\, \rho_n \nearrow 1$ and $\{\epsilon_n\},\; \epsilon_n > 0,\, \epsilon_n \searrow 0$, such that, for all $n \in \mathbb{N}\cup\{0\}$, 
$$|F_{n,m}^N(r,\theta)| < \epsilon_n,\; (r, \theta) \in [0, \rho_n) \times [0, 2\pi).$$
\end{theorem}

\section{Preliminaries}\label{sprelim}

\subsection{Various classical results on Bessel functions}

Throughout the paper, we make use of various classical estimates about the Bessel functions of the first kind. For the sake of completeness, we state them explicitly in this section, but refer the reader to the various references for their proof. We begin with Meissel expansions for $J_n(nz)$ (see \cite[p. 227]{Wa}):

\begin{proposition}\label{pFirstMeisselDev}\textnormal{(First Meissel development)}

\noindent For $n \in \mathbb{N} \cup \{0\}$ and $z \in (0,1)$, we have that
$$J_n(nz) = \frac{ (nz)^n e^(n\sqrt{1-z^2} - V_n) }{e^n n! (1-z^2)^{1/4}(1+ \sqrt{1-z^2})^n},$$
where $\displaystyle V_n = \frac{1}{24n}\left(\frac{2+3z^2}{(1-z^2)^\frac{3}{2}} + 2\right) + o\left(\frac{1}{n}\right)  \rightarrow 0$ as $n \rightarrow \infty$.
\end{proposition}

\begin{proposition}\label{pSecondMeisselDev}\textnormal{(Second Meissel development)}

\noindent Let $n \in \mathbb{N} \cup \{0\}$ and $z > 1$. We define $\beta \in (0, \frac{\pi}{2})$ such that $z = \sec \beta$. Then,

$$J_n(nz) = J_n(n \sec \beta) = \frac{\sqrt{2}\cos\left(n(\tan \beta - \beta) - \frac{\pi}{4}\right)}{\sqrt{\pi n \tan(\beta)}} + o\left(\frac{1}{n}\right).$$

\end{proposition}

For large values of $n$, we have the following classic asymptotic result, due to Jacobi (see \cite[p. 195]{Wa}):

\begin{proposition}\label{pJacobiAsymptoticJn}\textnormal{(Jacobi asymptotic formula)}

\noindent If $n^2 = o(x)$, we have that $$J_n(x) = \sqrt{\frac{2}{\pi x}} \left(\cos{\left(x - \frac{n\pi}{2} - \frac{\pi}{4}\right)} + O\left(\frac{4n^2-1}{x}\right)\right).$$
\end{proposition}

Another asymptotic form is due to Cauchy (see \cite[ p. 231]{Wa}):

\begin{proposition}\label{pCauchyAsymptoticJn}\textnormal{(Cauchy asymptotic formula)}

$$J_n(n) = \frac{\Gamma(\frac{1}{3})}{ 2^{\frac{1}{3}} 3^{\frac{1}{6}}\pi }n^{-\frac{1}{3}} + o\left(n^{-\frac{1}{3}}\right).$$

\end{proposition}

The Landau formula provides the following upper bound (see \cite{La}):

\begin{proposition}\label{pLandau}\textnormal{(Landau upper bound)}
 $$|J_n(x)|< b n^{-\frac{1}{3}},\; b=0,674885...$$
\end{proposition}

Krasikov proved in \cite{K} another upper bound:

\begin{proposition}\label{pKrasikov}\textnormal{(Krasikov upper bound)}
Let $\mu=(2n+1)(2n+3)$. We then have that
$$ J_n^2(x) \leq \dfrac{4\big(4x^2- (2n+1)(2n+5)\big)}{\pi\big((4x^2-\mu)^{3/2} - \mu\big)}, \forall n>\frac{1}{2}, \forall x > \dfrac{\sqrt{\mu + \mu^{2/3}}}{2}.$$
\end{proposition}

The Airy function of the first type is defined for $x \leq 0$ by 
\begin{equation*}\label{eqAiry}
Ai(-x) = \sqrt{x}\left(J_{1/3}\left(\frac{2}{3}x^{3/2}\right) + J_{-1/3}\left(\frac{2}{3}x^{3/2}\right)\right).
\end{equation*} 

We denote by $a_m$ the absolute value of the $m$-th negative zero of $Ai(x)$, i.e., $a_m > 0$ and $Ai(-a_m) =0$. We have the following well-known estimates for $a_m$ (see for example [AS, formulas 10.4.94, 10.4.105]):

\begin{proposition}\label{pAiryZero}
For large enough $m$, the absolute value of the $m$-th negative zero $a_m$ of $Ai(x)$ satisfies $$a_m = \left(\dfrac{3\pi m}{2}\right)^{2/3} + o(m^{2/3}).$$
\end{proposition}

\subsection{Derived technical results for Bessel functions}

The following result provides good estimates for the real positive roots of $J_n$:

\begin{proposition}\label{pEstimateBesselZeros}
Let $k_{n,m}$ be the $m$-th positive zero of $J_n$. Then, for large enough $m$, the following holds:
\begin{equation}\label{estimate_k_nm}
n + C_1m^{2/3}n^{1/3} < k_{n,m} < n + C_1 m^{2/3}n^{1/3} + C_2 m^{4/3}n^{-1/3},
\end{equation}
where $C_1 = \dfrac{(9\pi^2)^{1/3}}{2}$ and $C_2 = \dfrac{9(3\pi^4)^{1/3}}{40}$.

If, moreover, $m= n^\gamma$, then the above statement becomes
\begin{equation}\label{estimate_k_gamma}
n + C_1n^{\frac{1+2\gamma}{3}} < k_{n,m} < n + C_1 n^{\frac{1+2\gamma}{3}} + C_2 n^{\frac{4\gamma - 1}{3}}.
\end{equation}

\end{proposition}

\begin{proof}
From \cite{QW}, we have the following bounds for $k_{n,m}$:
$$ n + \left(\frac{1}{2}\right)^{1/3} a_m n^{1/3} < k_{n,m} < n + \left(\frac{1}{2}\right)^{1/3} a_m n^{1/3} + \frac{3}{10}\left(\frac{1}{2}\right)^{2/3} a_m^2 n^{-1/3}.$$

If $m$ is large enough, the result follows directly by applying Proposition \ref{pAiryZero}.
\end{proof}

The following proposition establishes an analogous result for the real positive zeros of $J'_n$.

\begin{proposition}\label{pEstimateBesselZerosDerivative}
Let $k'_{n,m}$ be the $m$-th positive zero of $J'_n$. Then, for large enough $m$, the following holds:
\begin{equation}\label{eNeumannknm}
n + C_1 (m-1)^{2/3}n^{1/3} < k'_{n,m} < n + C_1 m^{2/3}n^{1/3} + C_2 m^{4/3}n^{-1/3},
\end{equation}
where $C_1 = \dfrac{(9\pi^2)^{1/3}}{2}$ and $C_2 = \dfrac{9(3\pi^4)^{1/3}}{40}$.

If, moreover, $m= n^\gamma$, then the above statement becomes
\begin{equation}\label{ekgammaNeumann}
n + C_1n^{\frac{1+2\gamma}{3}} < k'_{n,m} < n + C_1 n^{\frac{1+2\gamma}{3}} + C_2 n^{\frac{4\gamma - 1}{3}}.
\end{equation}

\end{proposition} 

\begin{proof}
Start by noting that $k_{n,m-1} < k'_{n,m} < k_{n,m}$, as proved in \cite{OLBCl}. In order to obtain \eqref{eNeumannknm}, simply use \eqref{estimate_k_nm} on $k_{n,m-1}$ and on $k_{n,m}$. 
\end{proof}

\section{Proof of main results in the Dirichlet case}\label{sproofs}

\subsection{Proof of Theorem \ref{thm2}}

\begin{proof}
For $0 \leq \gamma < 1$, we set $m = n^\gamma$ so that $\displaystyle \frac{m}{n} \rightarrow 0$ as $n \rightarrow \infty$. Thus, we have $\displaystyle n^{\frac{1+2\gamma}{3}} = (m^2n)^\frac{1}{3}$ 
and $\displaystyle n^{\frac{4\gamma -1}{3}} = (m^4n^{-1})^\frac{1}{3}$. We divide  ($\ref{estimate_k_gamma}$) by $n$ and get:
$$1 + C_1\left(\frac{m}{n}\right)^\frac{2}{3}  < \frac{k_{n,m}}{n} < 1 + C_1 \left( \frac{m}{n} \right)^\frac{2}{3} + o\left(\left(\frac{m}{n}\right)^\frac{2}{3}\right),$$
whence $\displaystyle \frac{k_{n,m}}{n} = 1 + C_1 \left( \frac{m}{n} \right)^\frac{2}{3} + o\left(\left(\frac{m}{n}\right)^\frac{2}{3}\right)$. Letting $\alpha(n)$ be the right-hand-side
of the above equation, we can write $\displaystyle k_{n,m} = \alpha(n) n$ and $\alpha(n) \rightarrow 1$ as $n \rightarrow \infty$. In order to use Meissel's first development $\eqref{pFirstMeisselDev}$ for $J_n(nz)$, we let $z = \alpha(n) r$ so that $J_n(k_{n,m}r) = J_n(nz)$. 
\medskip
For $n$ large, we can use Stirling's formula $n! \sim \sqrt{2\pi n}\left(\frac{n}{e}\right)^n$ to rewrite the Meissel development as follows:
$$J_n(nz) \sim \frac{1}{\sqrt{2\pi n}(1-z^2)^\frac{1}{4}}e^{\left[n \left(\log z + \sqrt{1-z^2} - \log(1 + \sqrt{1-z^2})\right)\right]}.$$
Now, let $f(z) := \log z + \sqrt{1-z^2} - \log(1+\sqrt{1-z^2})$. We have $f(z) \rightarrow -\infty$ as $z \rightarrow 0$ and $f(z) \rightarrow 0$ as 
$z \rightarrow 1$. Also, $\displaystyle f'(z) = \frac{1 + \sqrt{1 - z^2} - z^2}{z(1 + \sqrt{1-z^2})}$, whence we deduce that $f(z)$ is strictly increasing and thus negative 
for $z \in (0,1)$.  
\medskip
The condition $ 0 < z = \alpha(n) r < 1$ which is required in order to use the above expansion is satisfied for $r \in \left(0, \frac{1}{\alpha(n)}\right)$, in which case $J_n(nz)$ decays exponentially to $0$ as
$n \rightarrow \infty$. 
\medskip
Now, recall that the actual normalized eigenfunctions are given by $$F_{n,m}(r, \theta) = \sqrt{\dfrac{2}{\pi}} \times \frac{J_n(k_{n,m}r)}{J_{n+1}(k_{n,m})}\cos(n\theta),$$ and note that the factors $$ \sqrt{\dfrac{2}{\pi}} \times \dfrac{1}{J_{n+1}(k_{n,m})}$$ do not depend on $r$. We recover the theorem in its original statement by setting $\rho_n := \frac{1}{\alpha(n)}$. This concludes the proof.
\end{proof}

\subsection{Proof of Lemma \ref{lemme421}}

\begin{proof}
We split the proof into two parts. We start by showing that if $\gamma < 1$, then $\phi(\gamma) \geq  \frac{1-\gamma}{6}$, using a simple geometrical argument (part A). We then prove that if $\gamma < 1$, then $\phi(\gamma) \geq  \frac{\gamma}{12}$, using an analytical approach (part B). Combining parts A and B yields the desired result.

$\bullet$ Part A. \\
Let $\gamma<1$.  We want to compute 
\begin{equation}\label{eqnumera}
||F_{n,m}(r,\theta)||_\infty = C \max_{r\in[0,1]} \left| \frac{J_n(k_{n,m}r)}{J_{n+1}(k_{n,m})}\right|.
\end{equation}
To do so, we use a similar notation as per the proof of Theorem \ref{thm2}. Let
\begin{equation}\label{alpha}
\alpha(n) = \frac{k_{n,m}}{n} = 1 + C_1 n^{\frac{2}{3} (\gamma-1)} + O(n^{\frac{4}{3} (\gamma-1)}).
\end{equation}
By Theorem \ref{thm2}, we know that such eigenfunctions do not decrease exponentially near the boundary, namely for all $r\in [\rho_n,1)$. The width of such annulus is given by $1-\rho_n$, where $1 - \rho_n = C_1 n^{\frac{2}{3} (\gamma-1)} +  o(n^{\frac{2}{3} (\gamma-1)}) $ using \eqref{alpha}. Indeed, we have that
	\begin{eqnarray*}
	1 - \rho_n & = & 1 - \dfrac{1}{\alpha_n} = 1 -\dfrac{1}{1+C_1 n^{\frac{2}{3} (\gamma-1)} +  O(n^{\frac{4}{3} (\gamma-1)}) } \\
	& = & \dfrac{C_1 n^{\frac{2}{3} (\gamma-1)} +  O(n^{\frac{4}{3} (\gamma-1)})}{1 +C_1 n^{\frac{2}{3} (\gamma-1)} +  O(n^{\frac{4}{3} (\gamma-1)})} \\
	& = &  C_1 n^{\frac{2}{3} (\gamma-1)} +  o(n^{\frac{2}{3} (\gamma-1)}).
	\end{eqnarray*} 
	   If $\mathcal{A}$ denotes the annulus described above, then using Theorem \ref{thm2}, we get
\begin{eqnarray*}
1 = \int_\mathbb{D} F_{n,m}^2 (r,\theta) dA & = & \int_{\mathbb{D}\setminus \mathcal{A}}  F_{n,m}^2 (r,\theta) dA + \int_\mathcal{A}  F_{n,m}^2 (r,\theta) dA \\ 
& = & O(e^{-an})+ \int_\mathcal{A}  F_{n,m}^2 (r,\theta) dA.
\end{eqnarray*}
Since we want to obtain a lower bound for the growth rate of such eigenfunctions on $\mathcal{A}$, we need to minimize $\sup_\mathcal{A} F_{n,m}$ under the constraint $$\int_A  F_{n,m}^2 (r,\theta) dA= 1 - O(e^{-an}).$$ Thus, we need to impose that $F_{n,m}(r,\theta)$ does not depend on $r$. Consequently, let $F_{n,m}(r,\theta)= C_{n,m}$ on $\mathcal{A}$, a constant that does not depend on $r$. We thus get that
\begin{eqnarray*}
 \int_\mathcal{A} C_{n,m}^2 dA & = &  1 + O(e^{-an}) \\
\implies C_{n,m}^2 &=& \frac{1 + O(e^{-an})}{\operatorname{Area}(\mathcal{A})}.
\end{eqnarray*}
In order to get an explicit formula in terms of both $n$ and $m$, we compute the area of $\mathcal{A}$:
\begin{eqnarray*}
\operatorname{Area}(\mathcal{A}) &=& \pi1^2 - \pi \rho_n^2 =   \pi (1- \rho_n)(2 - (1-\rho_n))  \\
& = &  \pi \Big(C_1 n^{\frac{2}{3} (\gamma-1)} +  o(n^{\frac{2}{3} (\gamma-1)})\Big) \Big(2 - \big(C_1 n^{\frac{2}{3} (\gamma-1)} +  o(n^{\frac{2}{3} (\gamma-1)})\big)\Big) \\
& = & 2 \pi C_1  n^{\frac{2}{3} (\gamma-1)} +  o(n^{\frac{2}{3} (\gamma-1)}).
\end{eqnarray*}
Thus, we get that
\begin{equation}
 C_{n,m} = \sqrt{\frac{1 + O(e^{-an})}{2 \pi C_1  n^{\frac{2}{3} (\gamma-1)} +  o(n^{\frac{2}{3} (\gamma-1)})}} = C_2 n^{\frac{1}{3} (1-\gamma)} + o(n^{\frac{1}{3} (1-\gamma)}).
\end{equation}

Therefore, if  $\gamma < 1$, the maximal growth rate of eigenfunctions on $\mathbb{D}$ is at least greater than $n^{\frac{1}{3} (1-\gamma)}$. Since $\lambda_{n,m}=k_{n,m}^2 $ and $k_{n,m}^2= (n + o(n))^2$, we get that
$$ ||F_{n,m}(r,\theta)||_\infty \geq C_2 \lambda^{\frac{1-\gamma}{6}} + o( \lambda^{\frac{1-\gamma}{6}} ).$$
This concludes the proof of part A.

 $\bullet$ Part B.\\
Let $\gamma<1$.  We want to compute \eqref{eqnumera}. If $r < \frac{1}{\alpha}$, Theorem \ref{thm2} gives us that $|F_{n,m}(r,\theta)| \to 0$ as $n \to \infty$. Thus, we treat the case $r \geq \frac{1}{\alpha}$ and we do so by using known asymptotics for Bessel functions. The following lemma gives us a formula for the denominator of \eqref{eqnumera} when $\gamma<1$.

\begin{lemma}\label{GammaLeq1Jn}
 If $\gamma < 1$, then we have that
$$J_{n+1}(k_{m,n}) = \frac{\cos\left((n+1)(\tan \beta' - \beta') - \frac{\pi}{4}\right)}{n^\frac{2+\gamma}{6}\sqrt{C + o(1)} },$$
where $\beta' = \arccos \frac{n+1}{n\alpha(n)}$ and $\displaystyle \alpha(n) = 1 + C_1 \left( \frac{m}{n} \right)^\frac{2}{3} + o\left(\left(\frac{m}{n}\right)^\frac{2}{3}\right).$
\end{lemma}
The proof of Lemma \ref{GammaLeq1Jn} can be found in Section \ref{ProofGammaLeq1Jn}. Also note that $J_{n+1}(k_{m,n}) \neq 0 \ \forall n$ since a zero of $J_n$ cannot be a zero of $J_{n+1}$. This is a well known result that was first conjectured by Bourget in 1866 and proven by Siegel in 1929 (see for instance \cite{Wa} for details). 

Using Proposition \ref{pCauchyAsymptoticJn} and Lemma \ref{GammaLeq1Jn}, for $r = \frac{1}{\alpha}$, we have that
\begin{eqnarray}\label{partbeq1} \notag
\dfrac{J_n(n)} {J_{n+1}(k_{n,m})} =  \left( \frac{\Gamma(\frac{1}{3})}{ 2^{\frac{1}{3}} 3^{\frac{1}{6}}\pi }n^{-\frac{1}{3}} + o\left(n^{-\frac{1}{3}} \right) \right) & \times & \left( \frac{\cos\left((n+1)(\tan \beta' - \beta') - \frac{\pi}{4}\right)}{n^\frac{2+\gamma}{6}\sqrt{C + o(1)} }\right)^{-1} \\
\implies \dfrac{J_n(n)} {J_{n+1}(k_{n,m})} & = & \dfrac{C_1 n^{\frac{\gamma}{6}}+o( n^{\frac{\gamma}{6}})}{\cos\left((n+1)(\tan \beta' - \beta') - \frac{\pi}{4}\right)}.
\end{eqnarray}

On the other hand, if $r > \frac{1}{\alpha}$, we use Proposition \ref{pLandau} to get 
\begin{equation}\label{partbeq2}
|J_n(k_{n,m}r)| < b n^{-1/3}.
\end{equation}  
Combining \eqref{partbeq2} and Proposition \ref{GammaLeq1Jn}, we have that
\begin{eqnarray}\label{partbeq3} \notag
\left| \dfrac{J_n(k_{n,m}r)} {J_{n+1}(k_{n,m})} \right|  <  b n^{-1/3}  &\times& \left( \frac{|\cos\left((n+1)(\tan \beta' - \beta') - \frac{\pi}{4}\right) |}{n^\frac{2+\gamma}{6}\sqrt{C + o(1)} }\right)^{-1} \\
 \implies \left| \dfrac{J_n(k_{n,m}r)} {J_{n+1}(k_{n,m})} \right| & < & \left| \dfrac{C_2n^{\frac{\gamma}{6}} + o(n^{\frac{\gamma}{6}})}{\cos\left((n+1)(\tan \beta' - \beta') - \frac{\pi}{4}\right)} \right|.
\end{eqnarray}
Therefore, combining \eqref{partbeq1}, \eqref{partbeq3} and Theorem \ref{thm2}, we get that
\begin{eqnarray*}
||F_{n,m}(r,\theta)||_\infty = \left| \dfrac{C_2n^{\frac{\gamma}{6}} + o(n^{\frac{\gamma}{6}}) + O\left(\frac{1}{n}\right)}{\cos\left((n+1)(\tan \beta' - \beta') - \frac{\pi}{4}\right)} \right|.
\end{eqnarray*}
Since $\frac{1}{\cos(x)} \geq 1, \forall x$, we obtain the following bound:
$$ ||F_{n,m}(r,\theta)||_\infty \geq C_2 n^{\frac{\gamma}{6}} + o(n^{\frac{\gamma}{6}}) + O\left(\frac{1}{n}\right).  $$
Since $\lambda_{n,m}=k_{n,m}^2 $ and $k^2_{n,m}= (n + o(n))^2$, we get that
$$ ||F_{n,m}(r,\theta)||_\infty \geq C_2 \lambda^{\frac{\gamma}{12}} + o( \lambda^{\frac{\gamma}{12}} ).$$
This concludes the proof of part B.
\end{proof}

\begin{remark}\label{remark1}
Because of the "bad" approximation $\frac{1}{\cos(x)} \geq 1, \forall x$, the bound obtained in part B of Lemma \ref{lemme421} is not sharp . 
\end{remark}

\subsection{ Proof of Proposition \ref{pgamma}}

\begin{proof}
 We want to compute \eqref{eqnumera}. According to Theorem \ref{thm2}, we only need to study the case $r \geq \frac{1}{\alpha}$. Combining Proposition \ref{pCauchyAsymptoticJn} and Proposition \ref{pLandau}, we get that 
$$\max_{r\in[0,1]} \left| J_n(k_{n,m}r) \right| = C n^{-1/3}.$$
Moreover, using Proposition \ref{pEstimateBesselZeros}, notice that $$k_{n,m} > n + C_1 m^{2/3} n^{1/3}> n^{\frac{1+2\gamma}{3}},$$ implying that $n^2 = o(k_{n,m})$ since $\gamma>3>\frac{5}{2}$. Therefore, it is possible to use Proposition \ref{pJacobiAsymptoticJn} to evaluate $J_{n+1}(k_{n,m})$, yielding that

\begin{equation}\label{eqzeroas}
J_{n+1}(k_{n,m})=  \sqrt{\frac{2}{\pi k_{n,m}}} \left( \cos\left( k_{n,m} - \frac{n\pi}{4} - \frac{\pi}{4}\right)+ O\left(\frac{n^2}{k_{n,m}} \right)\right) .
\end{equation}

For large enough $m$, zeros of \eqref{eqzeroas} yield the following approximation for $k_{n,m}$
\begin{equation}\label{eqEstimZeros}
k_{n,m} =  \left\{ 
  \begin{array}{l l}
    a\pi - \frac{\pi}{4} + O(\frac{n^2}{k_{n,m}})  \mbox{ if n is even} \\
    a\pi + \frac{\pi}{4} + O(\frac{n^2}{k_{n,m}}) \mbox{ if n is odd,}
   \end{array}  \right.
\end{equation}
where $a\in\mathbb{N}\cup\{0\}$.
In fact, there are at most $O(n^{\frac{5}{2}})$ zeros that may not take the form given in \eqref{eqEstimZeros}. Moreover, the last zero of those that do not satisfy \eqref{eqEstimZeros} takes the form $k_M = k_{n,O(n^{\frac{5}{2}})}$. Therefore, using Proposition \ref{pEstimateBesselZeros}, we get that
$$ n + C_1 n^{1/3} (n^{5/2})^{2/3} = n + O(n^2) < k_M < n + O(n^2) + O(n^3).$$

Since $\gamma >3$, we deduce that $k_M = o(m)$. We can therefore improve \eqref{eqEstimZeros} to obtain that
\begin{equation*}\label{eqEstimZeros2}
k_{n,m} =  \left\{ 
  \begin{array}{l l}
    C_{n,m}\pi - \frac{\pi}{4} + o(1)  \mbox{ if n is even} \\
     C_{n,m}\pi + \frac{\pi}{4} + o(1) \mbox{ if n is odd,}
   \end{array}  \right.
\end{equation*}
where $C_{n,m} = m + O(n) + k_M = m + o(m)$ and $C_{n,m} \in \mathbb{N}\cup\{0\}$. In short, we have that
\begin{equation}\label{eqEstimZeros2}
k_{n,m} = m \pi + o(m) \mbox{ if $\gamma > 3$.}
\end{equation}
Therefore, using \eqref{eqnumera}, \eqref{eqzeroas}, and \eqref{eqEstimZeros2}, we get that
\begin{eqnarray*}
||F_{n,m}||_\infty &=& C n^{-1/3} k_{n,m}^{1/2} + o ( k_{n,m}^{1/2}n^{-1/3}) \\
& = &C  k_{n,m}^{1/2} m^{-\frac{1}{3\gamma}} + o ( k_{n,m}^{1/2}m^{-\frac{1}{3\gamma}}) \\
& = & C  k_{n,m}^{1/2} (  k_{n,m}^{-\frac{1}{3\gamma}}  + o ( k_{n,m}^{-\frac{1}{3\gamma}} )) \\
& = & C  \lambda_{n,m}^{\frac{1}{4} - \frac{1}{6\gamma}} + o ( \lambda_{n,m}^{\frac{1}{4} - \frac{1}{6\gamma}}) .
\end{eqnarray*}

\end{proof}

\subsection{Proof of Lemma \ref{lemme441} }

\begin{proof}
 We want to compute \eqref{eqnumera}. As per the proof of Proposition \ref{pgamma}, we have that
\begin{equation}\label{eql}
\max_{r\in[0,1]} \left| J_n(k_{n,m}r) \right| = C n^{-1/3}.
\end{equation}
 Therefore, we need to understand the behavior of $J_{n+1}(k_{n,m})$ as $\gamma\geq1$. The fact that $\gamma>3$ is instrumental in the approach based on Jacobi asymptotic formula that is used in the proof of Proposition \ref{pgamma}. Instead, we use Krasikov upper bound stated in Proposition \ref{pKrasikov}. To do that, let $x=k_{n,m}$. Note that since $\gamma \geq 1$, we apply Proposition \ref{pEstimateBesselZeros} to get that
$$ x = k_{n,m} > (1 + C_1) n, $$
where $C_1 > 0$. Moreover, we have that
$$ \dfrac{\sqrt{\mu + \mu^{2/3}}}{2} = \dfrac{\sqrt{4n^2+8n+3 + (4n^2 + 8n+ 3)^{2/3}}}{2} = n + o(n),$$
yielding that $ x > \frac{\sqrt{\mu + \mu^{2/3}}}{2}$. Thus, it is possible to use Proposition \ref{pKrasikov} to get that
\begin{equation}\label{eqKrasikov1}
J_{n+1}^2(k_{n,m}) \leq \dfrac{ 4 \big( 4k_{n,m}^2 - (4n^2+20n+21) \big) }{ \pi \Big( \big(4k_{n,m}^2 - (4n^2+16n+15)\big)^{3/2} - \big( 4n^2 +16n +15 \big) \Big) }.
\end{equation}
Since $ k_{n,m} > (1 + C_1) n$ and $\gamma \geq 1$, we have that $k_{n,m} - n = \alpha k_{n,m} + o(k_{n,m})$, where $\alpha \in (0,1]$. Combining that fact with \eqref{eqKrasikov1}, we get that
\begin{eqnarray}\label{eqKrasikov2} \notag
J_{n+1}^2(k_{n,m})  & \leq & \dfrac{ 16 \big(\alpha k^2_{n,m} + o(k_{n,m}^2)\big) }{ \pi \Big( \big( 4(\alpha k^2_{n,m} + o(k_{n,m}^2))^{3/2} \big) - (4n^2 +16n+15) \Big)} \\ \notag
& = & \dfrac{2(\alpha_{n,m}^2 + o(k_{n,m}^2))}{\pi(\alpha^{3/2}k^3_{n,m} + o(k^3_{n,m}))} = \dfrac{2}{\pi \alpha^{3/2} k_{n,m} + o(k_{n,m})}; \\
  & | J_{n+1}(k_{n,m})  | &  \leq  \left( \dfrac{2}{\pi\alpha^{3/2}k_{n,m} + o(k_{n,m})}\right)^{1/2} = C_2 k^{-1/2}_{n,m} + o(k_{n,m}^{-1/2}).
\end{eqnarray}
Combining \eqref{eqKrasikov2} and \eqref{eql} with \eqref{eqnumera}, we get that
\begin{eqnarray*}
||F_{n,m}(r,\theta)||_{\infty} & \geq & C_2 k^{1/2}_{n,m} n^{-1/3} + o( k^{1/2}_{n,m} n^{-1/3}) \\
& > & C_2 k^{1/2}_{n,m} k_{n,m}^{-1/3} + o( k^{1/2}_{n,m} k_{n,m}^{-1/3}) \\
& = & C_2 k^{1/6}_{n,m} + o(k_{n,m}^{1/6}) = C_2 \lambda^{1/12} + o(\lambda^{1/12}),
\end{eqnarray*}
since $n^{-1/3} > k_{n,m}^{-1/3}$.

\end{proof}


\subsection{Lemma \ref{GammaLeq1Jn}}\label{ProofGammaLeq1Jn}

\begin{proof}
In order to use Proposition \ref{pSecondMeisselDev}, we let 
\begin{equation}\label{Geq1}
J_{n+1}(k_{n,m}) = J_{n+1}(\alpha n) = J_{n+1}\Big((n+1)\sec\beta'\Big), 
\end{equation}
where $\beta' = \arccos \left( \dfrac{n+1}{n\alpha(n)} \right)$ and $ \alpha(n) = 1 + C_1 \left( \frac{m}{n} \right)^\frac{2}{3} + o\left(\left(\frac{m}{n}\right)^\frac{2}{3}\right)$. Notice that
\begin{eqnarray*}
\tan^2\beta' = \dfrac{ 1 - \cos^2 \beta'}{\cos^2\beta'} & = & \dfrac{1}{\cos^2\beta'} - 1 = \left(\dfrac{n\alpha(n)}{n+1}\right)^2 - 1 \\
& \implies & \tan \beta' = \sqrt{ \left(\dfrac{n\alpha(n)}{n+1}\right)^2 - 1 }.
\end{eqnarray*}
Since  $ \alpha(n) = 1 + C_1 \left( \frac{m}{n} \right)^\frac{2}{3} + o\left(\left(\frac{m}{n}\right)^\frac{2}{3}\right)$, we get that
\begin{eqnarray}\label{Geq2}  \notag
 \sqrt{ \left(\dfrac{n\alpha(n)}{n+1}\right)^2 - 1 } & = &  \sqrt{ \left(\dfrac{n \big( 1 + C_1 \left( \frac{m}{n} \right)^\frac{2}{3} + o\left(\left(\frac{m}{n}\right)^\frac{2}{3}\right)\big)}{n+1}\right)^2 - 1 } \\ \notag
 & = & \sqrt{ \Big(1 + O\left(\frac{1}{n}\right) \Big) \Big(1+2C_1 \left(\frac{m}{n}\right)^{2/3} + o\left[\left(\frac{m}{n}\right)^{2/3}\right] \Big) - 1 } \\
 & = & \sqrt{ 2C_1 \left(\frac{m}{n}\right)^{2/3} + o\left(\frac{m}{n}\right)^{2/3} } = \sqrt{2C_1} \left(\frac{m}{n}\right)^{1/3} + o\left[ \left(\frac{m}{n}\right)^{1/3}\right].
\end{eqnarray}
Combining \eqref{Geq1}, \eqref{Geq2} and Proposition \ref{pSecondMeisselDev}, we get that
\begin{eqnarray*}
J_{n+1}(k_{n,m}) & = & \dfrac{\cos\left( (n+1) \cdot (\tan\beta' - \beta') - \frac{\pi}{4} \right)}{\sqrt{(n+1)\frac{\pi}{2}\tan\beta'}} + o\left(\frac{1}{n}\right) \\
 & = & \dfrac{\cos\left( (n+1) \cdot (\tan\beta' - \beta') - \frac{\pi}{4} \right)}{\sqrt{C_3 (n+1)\big((\frac{m}{n})^{1/3}+o(\frac{m}{n})^{1/3}\big)} }+ o\left(\frac{1}{n}\right) \\
  & = & \dfrac{\cos\left( (n+1) \cdot (\tan\beta' - \beta') - \frac{\pi}{4} \right)}{\sqrt{C_3 n^{\frac{\gamma+2}{3}} +o( n^{\frac{\gamma+2}{3}})\big)} }+ o\left(\frac{1}{n}\right)  ,
\end{eqnarray*}
which concludes the proof.
\end{proof}

\subsection{Proof of Theorem \ref{thm1}}

\begin{proof}
Combining Lemma \ref{lemme421}, Lemma \ref{lemme441} and Proposition \ref{pgamma} yields the two lower bounds in Theorem \ref{thm1}. The upper bound follows from \cite[Theorem 1]{Gr}.
\end{proof}

\section{Proof of Theorem \ref{thm3} for the Neumann case}\label{sproofsNeumann}

\subsection{Proof of the technical results leading to Theorem \ref{thm3}}
We start by proving Theorem \ref{thm4}.

\begin{proof}[Proof of Theorem \ref{thm4}]
The proof is analogous to the proof given for the Dirichlet case. The only major change is that we have to bound $k'_{n,m}/n$. To do so, we use \eqref{eNeumannknm} and obtain that
$$1 + C_1\left(\frac{m-1}{n}\right)^\frac{2}{3}  < \frac{k'_{n,m}}{n} < 1 + C_1 \left( \frac{m}{n} \right)^\frac{2}{3} + o\left(\left(\frac{m}{n}\right)^\frac{2}{3}\right),$$
leading to 
\begin{equation}\label{eKnmE}
 \frac{k'_{n,m}}{n} = 1 + C_1 \left(\frac{m}{n}\right)^{\frac{2}{3}} + o\big( \left(\frac{m}{n}\right)^{\frac{2}{3}} \big).
\end{equation}
We let $\alpha^\prime$ be the right-hand side of \eqref{eKnmE} and let $z^\prime = \alpha^\prime z$. The rest of the proof follows from what is done in the Dirichlet case.
\end{proof}

We are now ready to prove Lemma \ref{lNeumann}.

\begin{proof}[Proof of Lemma \ref{lNeumann}]To prove the lemma, we must investigate the growth rate of the following normalizing factor of $F^N_{n,m}$:
\begin{equation}\label{eGrowthNeumann}
\sqrt{\dfrac{1}{\left(1-\dfrac{n^2}{k^{\prime \quad 2}_{n,m}}\right)}} \cdot \dfrac{J_n(k'_{n,m} r)}{J_n(k'_{n,m})}.
\end{equation}
We start by showing that if $\gamma \geq 1$,  the expression $1-\dfrac{n^2}{k^{\prime \quad 2}_{n,m}} \to 1$ as $\lambda^N\to\infty$. To do so, we use \eqref{ekgammaNeumann} and the fact that $\gamma\geq1$ to get that
$$k_{n,m}^{\prime \quad 2} =  C^2_1 n^{2+\frac{4\gamma}{3}} + o(n^{2+\frac{4\gamma}{3}}),$$
implying that 
\begin{equation}\label{eConst}
1 - \dfrac{n^2}{k_{n,m}^{\prime \quad 2}} = \dfrac{n^2}{  C^2_1 n^{2+\frac{4\gamma}{3}} + o(n^{2+\frac{4\gamma}{3}} )}  \to 1,
\end{equation}
as $\lambda^N\to\infty$. Also, since $\gamma \geq 1$, we can see that
\begin{equation}\label{eAlex}
k_{n,m}^{\prime} > n.
\end{equation}
We now study the ratio of Bessel functions of \eqref{eGrowthNeumann}. Notice that using Landau uniform bound stated in Proposition \ref{pLandau} combined with Proposition \ref{pCauchyAsymptoticJn} yields the following:
\begin{equation}\label{eLandauN}
\max_{r\in[0,1)} |J_n(k'_{n,m}r)| = C_b n^{\frac{-1}{3}},
\end{equation}
as was the case in the Dirichlet setting. Using Proposition \ref{pKrasikov}, we get that
\begin{equation}\label{eKra}
|J_n(k'_{n,m})|= Ck_{n,m}^{\prime \quad \frac{-1}{2}} + o\left(k_{n,m}^{\prime \quad \frac{-1}{2}}\right).
\end{equation}
Combining \eqref{eConst}, \eqref{eLandauN} and \eqref{eKra} yields that
\begin{equation*}
||F^N_{n,m}||_\infty \geq C k_{n,m}^{\prime \quad \frac{1}{2}} n^{\frac{-1}{3}} + o(k_{n,m}^{\prime \quad \frac{1}{2}} n^{\frac{-1}{3}}).
\end{equation*}
Using \eqref{eAlex}, we get that
$$||F^N_{n,m}||_\infty \geq C  k_{n,m}^{\prime \quad \frac{1}{6}} + o( k_{n,m}^{\prime \quad \frac{1}{6}}) = C (\lambda^N)^{ \frac{1}{12}} + o((\lambda^N)^{\frac{1}{12}}),$$
which concludes the proof of the case where $\gamma \geq 1$.

Now, we focus on the case $\gamma < 1$. 

Using \eqref{ekgammaNeumann}, we have that
$$ k^\prime_{n,m} = n + C_1 n^{\frac{2\gamma+1}{3}} + o(n^{\frac{2\gamma+1}{3}}),$$
resulting in
$$  \dfrac{n^2}{k_{n,m}^{\prime \quad 2}} = \dfrac{n^2}{n^2+Cn^{\frac{2\gamma+2}{3}}+o(n^{\frac{2\gamma+2}{3}})}.$$
Thus, we have that
\begin{equation}\label{eFacteur}
\sqrt{\dfrac{1}{1- \dfrac{n^2}{k_{n,m}^{\prime \quad 2}}}} = \sqrt{\dfrac{1}{C n^{\frac{2\gamma-2}{3}} + o(n^{\frac{2\gamma-2}{3}})}} = C^\prime n^{\frac{1-\gamma}{3}} + o( n^{\frac{1-\gamma}{3}}).
\end{equation}

Since the right-hand side of \eqref{eFacteur} does not tend to $1$ when $\gamma<1$, the bound obtained for $\gamma  < 1$ differs from what we obtained in the Dirichlet case in Lemma \ref{lemme421}.

We now have to focus on understanding the growth rate of
$$ \dfrac{J_n(k'_{n,m} r)}{J_n(k'_{n,m})}.$$
To do so, we follow the steps given in the second part of the proof of Lemma \ref{lemme421}. We start by using Proposition \ref{pSecondMeisselDev} with
$$ J_n(k^\prime_{n,m}) = J_n(\alpha' n) = J_n(n \sec(\beta')), \beta'=\arccos\left(\frac{1}{\alpha'}\right),$$
where $\alpha'$ is defined in \eqref{eKnmE}. Since $\tan(\beta') = \sqrt{\alpha^{\prime\ 2} -1 }$, we have that
$$\tan(\beta') = \sqrt{\left(1 + C_1 \left(\frac{m}{n}\right)^{\frac{2}{3}} + o\big( \left(\frac{m}{n}\right)^{\frac{2}{3}} \big)\right)^2 -1} = \sqrt{2C} \left(\dfrac{m}{n}\right)^{\frac{1}{3}} + o\big(\left(\dfrac{m}{n}\right)^{\frac{1}{3}}\big).$$
The rest of the computations for that factor are done exactly as in the Dirichlet setting. It leads to the following estimate:
\begin{equation}\label{eBesselNeumann}
\sup_\mathbb{D} \left|  \dfrac{J_n(k'_{n,m} r)}{J_n(k'_{n,m})}  \right| \geq C (\lambda^N)^{  \frac{1}{12}} + o((\lambda^N)^{ \frac{1}{12}}).
\end{equation}
Combining \eqref{eFacteur} with \eqref{eBesselNeumann} leads to
\begin{equation*}
||F^N_{n,m}||_\infty \geq C \left( n^{\frac{1-\gamma}{3}} + o( n^{\frac{1-\gamma}{3}}) \right) \cdot \left(  (\lambda^N)^{ \frac{1}{12}} + o((\lambda^N)^{  \frac{1}{12}})\right).
\end{equation*}
Using \eqref{eKnmE} and the fact that $\gamma <1$, we get that
$$ n = k^\prime_{n,m} + o(k^\prime_{n,m}),$$
yielding the desired result, namely that
$$||F^N_{n,m}||_\infty \geq C (\lambda^N)^{ \frac{2-\gamma}{12}} + o((\lambda^N)^{  \frac{2-\gamma}{12}}).$$
\end{proof}

\subsection{Proof of Theorem \ref{thm3}}

\begin{proof}
Lemma \ref{lNeumann} yields the lower bound in Theorem \ref{thm3}. The upper bound follows from \cite[Theorem 1]{Gr}, which also holds for the Neumann case.
\end{proof}

\section{Numerical simulations}\label{Numeriques}
\subsection{The Dirichlet case}
To complete the second column of Table \ref{tab1}, we compute the values of \eqref{eAccumulationPts} via Mathematica 11.3. The third column shows the expected theoretical result (based on the results proved in this paper). The last column shows the conjectured result, where appropriate. 

\begin{table}[H]
\centering
\begin{tabular}{|c|c|c|c|}
\hline
$\gamma$ & numerical value & theoretical value  & conjectured value  \\ \hline\hline
$0$ &  $\frac{1}{6.08}$ &  $\frac{1}{6}$ & -  \\ \hline
$\frac{1}{2}$ &  $\frac{1}{9.96}$ & $\frac{1}{12}$ & - \\ \hline
$\frac{2}{3}$ &  $\frac{1}{11.5}$ & $\frac{1}{18}$ & - \\ \hline
$\frac{3}{4}$ &  $\frac{1}{13.03}$ & $\frac{1}{16}$ & -  \\ \hline
$1$ & $\frac{1}{11.95}$ & $\frac{1}{12}$  & $\frac{1}{12}$ \\ \hline
$\frac{3}{2}$ & $\frac{1}{7.19}$ & $\frac{1}{12}$ & $\frac{5}{36} = \frac{1}{7.2}$ \\ \hline
$\frac{7}{4}$ & $\frac{1}{6.45}$ & $\frac{1}{12}$ & $\frac{13}{84}= \frac{1}{6.46\ldots}$ \\ \hline
$2$ & $\frac{1}{5.99}$ & $\frac{1}{12}$ & $\frac{1}{6}$ \\ \hline
$4$ & $\frac{1}{4.78}$ & $\frac{5}{24} = \frac{1}{4.8}$ & . \\
\hline
\end{tabular}
\caption{Simulations in the Dirichlet case}
\label{tab1}
\end{table}

\subsection{The Neumann case}
The following table presents the results for the Neumann case.

\begin{table}[H]
\centering
\begin{tabular}{|c|c|c|c|}
\hline
$\gamma$ & numerical value & theoretical value & conjectured value \\ \hline\hline
$1$ & $\frac{1}{10.23}$ & $\frac{1}{12}$ & $-$ \\ \hline
$\frac{5}{4}$ & $\frac{1}{8.16}$ & $\frac{1}{12}$ & $\frac{7}{60} = \frac{1}{8.57}$  \\ \hline
$\frac{3}{2}$ & $\frac{1}{7.06}$ & $\frac{1}{12}$ & $\frac{5}{36} = \frac{1}{7.2}$ \\ \hline
$\frac{7}{4}$ & $\frac{1}{6.40}$ & $\frac{1}{12}$ & $\frac{13}{84}= \frac{1}{6.46\ldots}$ \\ \hline
$2$ & $\frac{1}{5.98}$ & $\frac{1}{12}$ & $\frac{1}{6}$ \\ \hline
$\frac{5}{2}$ & $\frac{1}{5.46}$ & $\frac{1}{12}$ & $\frac{11}{60} = \frac{1}{5.45}$ \\ \hline
$3$ & $\frac{1}{5.16}$ & $\frac{1}{12}$ & $\frac{7}{36} = \frac{1}{5.14}$ \\ \hline
$4$ & $\frac{1}{4.83}$ & $\frac{1}{12}$ & $\frac{5}{24} = \frac{1}{4.8}$  \\ \hline
\end{tabular}
\caption{Simulations in the Neumann case}
\label{tab2}
\end{table}

\vspace{3 mm}

\noindent{ Guillaume Lavoie, }

\emph{E-mail address:}   \verb"guillaume.lavoie.87@gmail.com" \\

\noindent{Guillaume Poliquin, \\
\scshape Département de mathématiques,
Collège Ahuntsic, 9155 rue Saint-Hubert, Montréal,
H2M 1Y8, Québec, Canada.}

\emph{E-mail address:}   \verb"guillaume.poliquin@collegeahuntsic.qc.ca"

\end{document}